\documentclass[11pt, a4paper]{amsart}
\usepackage{a4, a4wide}
\usepackage{amssymb}
\usepackage{amsmath}
\usepackage{amsthm}
\usepackage{amstext}
\usepackage{amscd}
\usepackage{latexsym}
\usepackage{graphics}
\usepackage{color}
\usepackage{enumitem}
\usepackage[all]{xy}
\usepackage[textsize=scriptsize]{todonotes}

\usepackage[colorlinks, pagebackref]{hyperref}
\hypersetup{
  colorlinks=true,
  citecolor=blue,
  linkcolor=blue,
  urlcolor=blue}
\makeatletter
\def\@tocline#1#2#3#4#5#6#7{\relax
  \ifnum #1>\c@tocdepth % then omit
  \else
    \par \addpenalty\@secpenalty\addvspace{#2}%
    \begingroup \hyphenpenalty\@M
    \@ifempty{#4}{%
      \@tempdima\csname r@tocindent\number#1\endcsname\relax
    }{%
      \@tempdima#4\relax
    }%
    \parindent\z@ \leftskip#3\relax \advance\leftskip\@tempdima\relax
    \rightskip\@pnumwidth plus4em \parfillskip-\@pnumwidth
    #5\leavevmode\hskip-\@tempdima
      \ifcase #1
      \or\or \hskip 2em \or \hskip 2em \else \hskip 3em \fi%
      #6\nobreak\relax
    \dotfill\hbox to\@pnumwidth{\@tocpagenum{#7}}\par
    \nobreak
    \endgroup
  \fi}
\makeatother

\theoremstyle{plain}

\newtheorem{theorem}{Theorem}[section]

\newtheorem{lemma}[theorem]{Lemma}
\newtheorem{corollary}[theorem]{Corollary}
\newtheorem{proposition}[theorem]{Proposition}

\theoremstyle{definition}

\newtheorem{remark}[theorem]{Remark}

\newtheorem{definition}[theorem]{Definition}

\numberwithin{equation}{section}
\newcommand{\ilim}{\mathop{\varprojlim}\limits} % inverse limit
\newcommand{\dlim}{\mathop{\varinjlim}\limits}  % direct limit

\newcommand{\codim}{{\rm codim}}

\newcommand{\coker}{{\rm Coker}}

\newcommand{\im}{{\rm Im}}
\newcommand{\Spec}{{\rm Spec \,}}

\newcommand{\A}{{\mathbb A}}
\renewcommand{\P}{{\mathbb P}}

\newcommand{\Sing}{{\rm Sing}_*^{\mathbb A^1}}

\def\<{\langle}
\def\>{\rangle} 
\def\-{\overline} 
\def\~{\widetilde}
\def\^{\widehat}

\def\@{\mathcal}
\def\!{\mathscr}
\def\#{\mathbb}
\def\_{\underline}

\input{xy}
\xyoption{all}
\begin{document}

\title[Strong $\A^1$-invariance]{Strong $\A^1$-invariance of $\A^1$-connected components of reductive algebraic groups}

\author{Chetan Balwe}
\address{Department of Mathematical Sciences, Indian Institute of Science Education and Research Mohali, Knowledge City, Sector-81, Mohali 140306, India.}
\email{cbalwe@iisermohali.ac.in}

\author{Amit Hogadi}
\address{Department of Mathematical Sciences, Indian Institute of Science Education and Research Pune, Dr. Homi Bhabha Road, Pashan, Pune 411008, India.}
\email{amit@iiserpune.ac.in}

\author{Anand Sawant}
\address{School of Mathematics, Tata Institute of Fundamental Research, Homi Bhabha Road, Colaba, Mumbai 400005, India.}
\email{asawant@math.tifr.res.in}
\date{}
\thanks{The authors acknowledge the support of India DST-DFG Project on Motivic Algebraic Topology DST/IBCD/GERMANY/DFG/2021/1, SERB Start-up Research Grant SRG/2020/000237 and the Department of Atomic Energy, Government of India, under project no. 12-R\&D-TFR-5.01-0500.}
%\subjclass[2010]{}
%\keywords{}

\begin{abstract} 
We show that the sheaf of $\A^1$-connected components of a reductive algebraic group over a perfect field is strongly $\A^1$-invariant.  As a consequence, torsors under such groups give rise to $\A^1$-fiber sequences.  We also show that sections of $\A^1$-connected components of anisotropic, semisimple, simply connected algebraic groups over an arbitrary field agree with their $R$-equivalence classes, thereby removing the perfectness assumption in the previously known results about the characterization of isotropy in terms of affine homotopy invariance of Nisnevich locally trivial torsors.
\end{abstract}

\maketitle
\tableofcontents

\setlength{\parskip}{4pt plus1pt minus1pt}
\raggedbottom

\section{Introduction}
\label{section introduction}

Principal bundles (or in other words, torsors) under a group are among the most typical examples of fibrations in topology.  The concept of a fibration is extremely useful, since a fibration gives rise to a long exact sequence consisting of homotopy groups of the total space, the base space and the fiber.  Motivic homotopy theory developed by Morel and Voevodsky \cite{Morel-Voevodsky} provides a framework to import techniques from topology into algebraic geometry.  The first application of motivic fiber sequences is the theory of Euler class developed by Morel in \cite{Morel-book}. Computations using motivic fiber sequences have later been successfully used in the works of Asok and Fasel \cite{Asok-Fasel-3folds, Asok-Fasel-4folds} to address classical questions about splitting properties of vector bundles on smooth affine varieties using the $\A^1$-homotopical interpretation of these questions.  Unlike in classical topology, torsors under sheaves of groups do not automatically give rise to fiber sequences in $\A^1$-homotopy theory, the reason behind which can be attributed to the peculiar behaviour of the sheaf of $\A^1$-connected components.  The set of connected components of a space in classical topology is a discrete space.  The analogous statement in $\A^1$-homotopy theory states that the sheaf of $\A^1$-connected components of a space is $\A^1$-invariant.  This statement, conjectured to be true by Morel \cite{Morel-book}, is known to hold in many cases, but has been recently disproved in its generality by Ayoub \cite{Ayoub-counterexample}.

Let us fix a base field $k$ and denote by $\@H(k)$ the unstable motivic homotopy category over $k$, the objects of which are simplicial Nisnevich sheaves of sets on the category $Sm_k$ of smooth schemes over $k$.  We refer the reader to Sections \ref{section birationality} and \ref{section preliminaries alg groups} for relevant definitions not mentioned in the introduction.  Let $G$ be a Nisnevich sheaf of groups on $Sm_k$ and let $\@P \to \@X$ be a Nisnevich locally trivial $G$-torsor in $\@H(k)$.  This means that $\@P$ has a right $G$-action, the morphism $\@P \to \@X$ is equivariant for the trivial right $G$-action on $\@X$ and there is an isomorphism $G \times \@P \stackrel{\sim}{\to} \@P \times_{\@X} \@P$ of simplicial Nisnevich sheaves.  By a result of Choudhury \cite{Choudhury}, it is known that $\pi_0^{\A^1}(G)$ is $\A^1$-invariant.  Morel showed in \cite[Theorem 6.50]{Morel-book} that if the sheaf $\pi_0^{\A^1}(G)$ of $\A^1$-connected components of $G$ is strongly $\A^1$-invariant, then a $G$-torsor $\@P \to \@X$ yields an $\A^1$-fiber sequence
\[
\mathcal P \to \mathcal X \to BG. 
\]
We will be interested in the case when $G$ is represented by an algebraic group over $k$.  Strong $\A^1$-invariance of $\pi_0^{\A^1}(G)$ is known when $G$ is isotropic, semisimple and simply connected over a field due to the work of Asok, Hoyois and Wendt \cite{AHW2, AHW3}.  The main result of this paper is to generalize this result to all reductive algebraic groups, at least when the base field is perfect.

\begin{theorem}[Theorem \ref{theorem reductive} and Corollary \ref{cor a1 fiber sequence}]
\label{theorem intro main}
If $G$ is a reductive algebraic group over a perfect field $k$, then $\pi_0^{\A^1}(G)$ is strongly $\A^1$-invariant.  Consequently, for any Nisnevich locally trivial $G$-torsor $\mathcal P \to \mathcal X$ in $\mathcal H(k)$, 
\[
\mathcal P \to \mathcal X \to BG 
\]
is an $\A^1$-fiber sequence.
\end{theorem}

\begin{remark}
It is easy to see that the conclusion of Theorem \ref{theorem intro main} holds for any algebraic group over a perfect field.  Indeed, strong $\A^1$-invariance of $\pi_0^{\A^1}(G)$ for an arbitrary algebraic group $G$ over a perfect field reduces to the case of affine linear algebraic group by Chevalley's structure theorem (see \cite[Section 5.2]{Wendt}, for example), which says that any algebraic group $G$ over a perfect base field $k$ admits a linear algebraic subgroup $H$ such that the quotient $G/H$ is an abelian variety.  Since abelian varieties are $\A^1$-rigid, it follows that $\pi_0^{\A^1}(H) \simeq \pi_0^{\A^1}(G)$.  One then reduces further to connected linear algebraic groups by \cite[Proposition 5.6]{Wendt}.  The quotient of any connected linear algebraic group by its unipotent radical is a reductive group.  One can reduce to the case of reductive groups by observing that the unipotent radical is $\A^1$-contractible over a perfect field, being isomorphic to an affine space (see \cite[Remark A.3]{Kambayashi-Miyanishi-Takeuchi}, for example). 
\end{remark}

There are two main ingredients in the proof of Theorem \ref{theorem intro main}.  The first one states that if $G$ as in Theorem \ref{theorem intro main} is moreover semisimple and simply connected, then $\pi_0^{\A^1}(G)$ is a birational sheaf (see Theorem \ref{theorem birational}).  The perfectness assumption on $k$ is needed here since our arguments require the study of base change of $G$ to the coefficient field of a smooth henselian local ring, for which we require the coefficient field to contain a copy of $k$.  Note that even in the case of isotropic, semisimple, simply connected and absolutely almost simple $G$, Theorem \ref{theorem birational} provides a refinement of the result in \cite{AHW2, AHW3} about the strong $\A^1$-invariance of $\pi_0^{\A^1}(G)$.  The second key ingredient in the proof of Theorem \ref{theorem intro main} is the preservation of strong $\A^1$-invariance under isogenies of simply connected groups (see Lemma \ref{lemma isogeny}).  This step relies crucially on a tricky result of Choudhury and Hogadi \cite[Theorem 1.3]{Choudhury-Hogadi}, which states that $\A^1$-invariant quotients of strongly $\A^1$-invariant sheaves are strongly $\A^1$-invariant.  Another key input in this step is the $\A^1$-homotopical interpretation of classical results on torsors on the affine line in terms of $\A^1$-connected components of classifying spaces. 

We believe that it should be possible to prove Theorem \ref{theorem intro main} over an arbitrary field.  The second main result of this paper (Theorem \ref{theorem anisotropic R-equivalence}) concerns the comparison of $R$-equivalence classes and sections of $\pi_0^{\A^1}(G)$ over arbitrary fields, generalizing \cite[Main Theorem]{Balwe-Sawant-R-eqivalence-IMRN}.  This leads to the anisotropic analogue of the reinterpretation of \cite[Th\'eor\`eme 5.8, Th\'eor\`eme 7.2]{Gille} in \cite[Corollary 4.3.6]{AHW2}.  

\begin{theorem}[Theorem \ref{theorem anisotropic R-equivalence}, Corollary \ref{cor R-equivalence} and Corollary \ref{cor contraction}]
\label{theorem intro R-equivalence}
Let $k$ be an arbitrary field and let $G$ be a semisimple, simply connected algebraic group over $k$.  Then the natural map $$\pi_0^{\A^1}(G)(k) \to G(k)/R$$ is a bijection.  Moreover, the contraction $\pi_0^{\A^1}(G)_{-1}$ is the trivial point sheaf.
\end{theorem}
Theorem \ref{theorem intro R-equivalence} vindicates the principle that for semisimple, simply connected groups, $\A^1$-connected components provide the right framework as an anisotropic analogue of the Whitehead group of an isotropic group and can be used to put results about isotropic and anisotropic groups on the same footing.

\subsection*{Acknowledgements}
We thank Matthias Wendt for helpful correspondence regarding the results of \cite{EKW}, Swarnava Mukhopadhyay for pointing us to the literature on principal bundles on curves and the anonymous referee for very helpful comments.

\subsection*{Notations and conventions}
The letter $k$ will always denote a field.  We will denote by $Sm_k$ the big Nisnevich site of finite type, separated schemes over $k$.  Given a simplicial sheaf of sets $\@X$ on $Sm_k$ and an affine smooth scheme $\Spec A$ over $k$, we will often write $\@X(A)$ for $\@X(\Spec A)$ for the sake of brevity.   

We will often make use of essentially smooth schemes, that is, schemes which are filtered inverse limits of diagrams of smooth schemes in which the transition maps are \'etale, affine morphisms.  All presheaves on $Sm_k$ will be extended to essentially smooth schemes by defining $\@F(\ilim U_{\alpha}) = \dlim \@F(U_\alpha)$.   By a henselian local scheme over $k$, we mean an affine scheme of the form $U = \Spec(R)$, where $R$ is a henselian local $k$-algebra.

\section{Birationality and Strong \texorpdfstring{$\A^1$}{A1}-invariance}
\label{section birationality}

In this section, we recall definitions relevant to birationality and strong $\A^1$-invariance of sheaves on $Sm_k$ and basic properties that will used the proofs of our main results.

\begin{definition}(\cite[Definition 1.7]{Morel-book})
\label{defn a1-invariance}
Let $\@F$ be a presheaf of sets on $Sm_k$.
\begin{itemize}
\item We say that $\@F$ is \emph{$\A^1$-invariant} if for every $U \in Sm_k$, the map $ \@F(U)\to \@F(U\times \A^1)$ induced by the projection $U\times \A^1 \to U$, is a bijection.

\item We say that $\@F$ is \emph{strongly $\A^1$-invariant} if for every $U \in Sm_k$, the map $ H^i_{\rm Nis}(U, \@F) \to H^i_{\rm Nis}(U\times \A^1, \@F)$ induced by the projection $U\times \A^1 \to U$, is a bijection for $i=0,1$.
\end{itemize}
\end{definition}

\begin{definition}(\cite[Definition 6.1.1]{Asok-Morel})
\label{defn birational sheaf}
A presheaf $\@F$ of sets on $Sm_k$ is said to be \emph{birational} if 
\begin{enumerate}
\item for every $X \in Sm_k$ with irreducible components $X_1, \ldots, X_n$, the map 
\[
\@F(X) \to \prod_{1 \leq i \leq n} \@F(X_i) 
\]
is a bijection; and
\item for every $X \in Sm_k$ and a dense open subscheme $U$ of $X$, the restriction map $\@F(X) \to \@F(U)$ is a bijection.
\end{enumerate}
A sheaf of sets is said to be birational if its underlying presheaf is birational.
\end{definition}

\begin{remark}
\label{rmk birationality and strong a1-invariance} 
Let $\@F$ be a birational presheaf of (not necessarily abelian) groups on $Sm_k$.  Then $\@F$ is already a Nisnevich sheaf by \cite[Lemma 6.1.2]{Asok-Morel}.  Note that $\@F$ is flasque by definition.  It is well-known that the Nisnevich cohomology $H^1_{\rm Nis}(X, \@F)$ vanishes for all $X \in Sm_k$; we include a proof here for the convenience of the reader.  Take an element of $H^1_{\rm Nis}(X, \@F)$, which is represented by a Nisnevich locally trivial $\@F$-torsor $P$ on $X$.  Let $U$ be a maximal open subset of $X$ on which this torsor is trivial.  Suppose, if possible, that $U \neq X$.  Let $\eta$ be a generic point of $X \setminus U$.  There exists an \'etale map $V \to X$ such that $U \coprod V$ is an elementary Nisnevich cover of of an open subset $U'$ of $X$ containing $U$ and $\eta$ and such that the torsor $P$ is trivial on $V$.  Since $\@F(V) \to \@F(U \times_{U'} V)$ is surjective, a straightforward \v{C}ech cohomology calculation shows that $P$ is trivial on $U'$, contradicting the choice of $U$.  Hence, $H^1_{\rm Nis}(X, \@F) = 0$, for every $X \in Sm_k$.  Thus, we conclude that if $\@F$ is a birational and $\A^1$-invariant sheaf of groups, then it is strongly $\A^1$-invariant. 
\end{remark}

Birationality of a Nisnevich sheaf of groups on $Sm_k$ can be checked by verifying the condition given in Definition \ref{defn birational sheaf} on smooth henselian schemes.

\begin{lemma}
\label{lemma henselian reduction}
Let $\mathcal F$ be an $\mathbb{A}^1$-invariant sheaf of groups on $Sm_k$.  Suppose that for every smooth henselian local $k$-algebra $R$ with quotient field $K$, the restriction map $\mathcal F(R) \to \mathcal F(K)$ is an isomorphism.  Then $\mathcal F$ is a birational sheaf.
\end{lemma}
\begin{proof}
Let $X \in Sm_k$ and let $U$ be a dense open subset of $X$.  By \cite[Lemma 3.45]{Morel-book}, we know that the restriction map $\mathcal F(X) \to \mathcal F(U)$ is injective.  We now show that it is surjective.

Let $\alpha \in \mathcal F(U)$.  For every point $x \in X$, we have a commutative diagram
\[
\begin{xymatrix}{
\mathcal F(X) \ar[r] \ar[d] & \mathcal F(U) \ar[d] \\
\mathcal F(\mathcal O_{X, x}^h) \ar[r] & \mathcal F(K_x^h)}
\end{xymatrix} 
\]
in which $K_x^h$ denotes the quotient field of henselization $\mathcal O_{X, x}^h$ of the local ring $\mathcal O_{X, x}$.  Since the bottom horizontal arrow is an isomorphism by hypothesis, we can lift $\alpha \in \mathcal F(U)$ to an element of $\mathcal F(\mathcal O_{X, x}^h)$, which by a standard limiting argument lifts to an element $\alpha_x \in \mathcal F(V_x)$ for some Nisnevich neighborhood $V_x$ of $x$.  Since $X$ is noetherian, we can find a refinement $\~X \to X$ of the Nisnevich cover $\coprod_{x \in X} V_x \to X$ with $\~X \in Sm_k$ such that $\alpha \in \mathcal F(U)$ lifts to an element $\~\alpha \in \mathcal F(\~X)$.  We now show that $\~\alpha$ descends to an element of $\mathcal F(X)$ which restricts to $\alpha$.  Let $\~U$ denote the inverse image of $U$ under the map $\~X \to X$. We now have a commutative diagram 
\[
\begin{xymatrix}{
\mathcal F(X) \ar[r] \ar[d]& \mathcal F(\~X) \ar@<-.5ex>[r] \ar@<.5ex>[r] \ar[d] & \mathcal F(\~X \times_X \~X) \ar[d]\\ 
\mathcal F(U) \ar[r] & \mathcal F(\~U) \ar@<-.5ex>[r] \ar@<.5ex>[r] & \mathcal F(\~U \times_U \~U) }
\end{xymatrix}
\]
in which the rows are equalizer diagrams and all the vertical maps are injective.  It remains to observe that the image of $\~\alpha$ under the two maps $F(\~X) \rightrightarrows \mathcal F(\~X \times_X \~X)$ is the same.  But this follows from the injectivity of the rightmost vertical map and the fact that the restriction of $\~\alpha$ to $\~U$ is the image of $\alpha \in \mathcal F(U)$.  This shows that $\mathcal F(X) \to \mathcal F(U)$ is surjective, completing the proof.
\end{proof}

Since the sheaf of $\A^1$-connected components of any $H$-group is $\A^1$-invariant by the results of \cite{Choudhury}, the restriction maps on such sheaves happen to be injective.  This observation will prove useful in our analysis of $\A^1$-connected components of a reductive algebraic group.

\begin{lemma}
\label{lemma pi0 injective}
Let $\mathcal G$ be a sheaf of groups on $Sm_k$.  Then for any $X \in Sm_k$ and a dense open subset $U$ of $X$, the restriction map $\pi_0^{\A^1}(\mathcal G)(X) \to \pi_0^{\A^1}(\mathcal G)(U)$ is injective. In particular, the induced map $\pi_0^{\A^1}(\mathcal G)(X) \to \pi_0^{\A^1}(\mathcal G)(k(X))$ is injective.
\end{lemma}
\begin{proof}
This is a straightforward consequence of \cite[Theorem 4.18]{Choudhury} and \cite[Lemma 3.45]{Morel-book}. 
\end{proof}

\section{\texorpdfstring{$\A^1$}{A1}-connected components of reductive groups and their classifying spaces} 
\label{section preliminaries alg groups}

One of our main interests behind the study of the sheaf of $\A^1$-connected components of a reductive group $G$ over a field is its application to $\A^1$-fiber sequences due to Morel \cite[Theorem 6.50]{Morel-book} (see also \cite[Theorem 3.5]{Asok-2013-JTop}).  We refer the reader to \cite[Chapter 6]{Morel-book} or \cite[Section 2]{Asok-2013-JTop} for generalities on $\A^1$-fiber sequences.  In this section, we study the sheaf of $\A^1$-connected components of reductive algebraic groups and their classifying spaces. 

\subsection{Reductive groups} \hfill
\label{subsection reductive groups}

We refer the reader to \cite{Milne-algebraic groups} and \cite[Appendix A]{Conrad-Gabber-Prasad} for basic facts from algebraic group theory.  A semisimple, simply connected algebraic group $G$ over a field $k$ is called \emph{isotropic} if every almost simple factor of $G$ contains a non-central $k$-subgroup-scheme isomorphic to $\mathbb G_m$.  A semisimple, simply connected algebraic group is \emph{anisotropic} if it contains no subgroup isomorphic to $\mathbb G_m$.

Let $G$ be a semisimple, simply connected and absolutely almost simple algebraic group over an arbitrary field $k$.  The \emph{Whitehead group} of $G$ over $k$ is defined to be 
\[
W(k, G):=G(k)/G(k)^+, 
\]
where $G(k)^+$ is the normal subgroup of $G(k)$ generated by $k$-rational points of subgroups of $G$ isomorphic to $\mathbb G_a$.  A related invariant of $G$ is the group $G(k)/R$ of \emph{$R$-equivalence} classes, defined to be the quotient of $G(k)$ by the normal subgroup consisting of elements $g \in G(k)$ for which there exists a rational map $h : \P^1 \dashrightarrow G$ defined at $0$ and $1$ such that $h(0)=e$ and $h(1)=g$, where $e$ denotes the neutral element of $G(k)$.  See \cite{Gille} for a detailed discussion of Whitehead groups and $R$-equivalence classes.  As a consequence of affine homotopy invariance of Nisnevich locally trivial $G$-torsors for isotropic groups $G$ (see \cite[Theorem 3.3.7]{AHW2} for the statement over infinite fields and \cite[Theorem 2.4]{AHW3} for the statement over finite fields), Asok, Hoyois and Wendt have shown that the Morel-Voevodsky singular construction $\Sing G$ is $\A^1$-local (see \cite[Remark 3.3.8]{AHW2}).  Thus, if $G$ is in addition isotropic, the natural map $W(k, G) \to G(k)/R$ factors through $\pi_0^{\A^1}(G)(k)$ and induces bijections
\begin{equation}
\label{eq isotropic R-eqivalence}
W(k, G) \xrightarrow{\simeq} \pi_0^{\A^1}(G)(k) \xrightarrow{\simeq} G(k)/R,  
\end{equation}
by \cite[Th\'eor\`eme 7.2]{Gille}.

In the case $G$ is anisotropic, the group $G(k)^+$ is trivial.  However, it was shown in \cite{Balwe-Sawant-R-eqivalence-IMRN} (see also \cite[Theorem 3.6]{Balwe-Sawant-reductive}) that over an infinite perfect field $k$, there is a bijection 
\[
\pi_0^{\A^1}(G)(k) \xrightarrow{\simeq} G(k)/R. 
\]
We first observe that the proof of this result given in \cite{Balwe-Sawant-R-eqivalence-IMRN} can be appropriately modified to make it work over an arbitrary base field.  There are two key ingredients in the proof - the first is the existence of compactifications of algebraic groups over an arbitrary field due to Gabber (see \cite[Th\'eor\`eme 5.2]{GGMB} and also \cite[Theorem 4, Theorem 6]{Gille-owr}), which yields the following consequence.

\begin{lemma}
\label{lemma compactification}
Let $G$ be an anisotropic semisimple group over an arbitrary field $k$.  Then any rational map $h: \P^1_k \dashrightarrow G$ is defined at all the $k$-rational points of $\P^1_k$.
\end{lemma}
\begin{proof}
Follow exactly the proof of \cite[Lemma 3.6]{Balwe-Sawant-R-eqivalence-IMRN} with \cite[8.2]{Borel-Tits} replaced by \cite[Th\'eor\`eme 5.2]{GGMB} (or \cite[Theorem 4, Theorem 6]{Gille-owr}).
\end{proof}

The second key ingredient is the choice of a Nisnevich cover of the affine line over which an $R$-equivalence between two $k$-points can be completed to give an $\A^1$-ghost homotopy connecting the two points.  This is where the proof of Theorem \ref{theorem anisotropic R-equivalence} differs from the one given in \cite{Balwe-Sawant-R-eqivalence-IMRN}.

\begin{theorem}
\label{theorem anisotropic R-equivalence}
Let $k$ be an arbitrary field and let $G$ be a semisimple, simply connected, absolutely almost simple and anisotropic group over $k$.  Then the natural map $$\pi_0^{\A^1}(G)(F) \to G(F)/R$$ is a bijection, for every field extension $F$ of $k$. 
\end{theorem}
\begin{proof}
We closely follow the proof of \cite[Theorem 4.2]{Balwe-Sawant-R-eqivalence-IMRN}.  We are reduced to the case $F=k$ as in \cite{Balwe-Sawant-R-eqivalence-IMRN}.  By \cite[8.2]{Borel-Tits} (in characteristic $0$) or \cite[Th\'eor\`eme 5.2]{GGMB} (in arbitrary characteristic), there exists a compactification $\-G$ of $G$ such that $G(k) = \-G(k)$.  If two elements $p$ and $q$ of $G(k)$ map to the same element in $\pi_0^{\A^1}(G)(k)$, then they map to the same element in $\pi_0^{\A^1}(\-G)(k)$ by functoriality.  Since $\-G$ is proper over $k$, we see that $p$ and $q$ map to the same element in $\mathcal S(\-G)(k)$ by \cite[Theorem 2.4.3]{Asok-Morel}.  Therefore, $p$ and $q$ are $\A^1$-chain homotopic $k$-rational points of $\-G$ and hence, $R$-equivalent in $G(k) = \-G(k)$.  In other words, the natural surjective map $G(k) \to G(k)/R$ factors through $\pi_0^{\A^1}(G)(k)$.

We now show the injectivity of the map $\pi_0^{\A^1}(G)(k) \to G(k)/R$.  Let $p$ and $q$ be two elements of $G(k)$, which are $R$-equivalent through a rational map $h: \P^1_{k} \dashrightarrow G$ defined on $0$ and $1$ such that $h(0) = p$ and $h(1) = q$.  Choose a compactification $\-G$ of $G$ such that $G(k) = \-G(k)$ and uniquely extend the rational map $h$ to a morphism $\-h: \P^1_{k} \to \-G$.  Now, $h$ is not defined only at points of $\A^1_{k}$ having residue fields that are non-trivial finite extensions of $k$.  We define $V := \-h^{-1}(G) \cap \A^1_k$ which is a Zariski open subscheme of $\A^1_k$. Let $\A^1_k \backslash V = \{p_1, \ldots, p_n\}$ and let the residue field at $p_i$ be $L_i$. We define $h_V: V \to G$ by $h_V := \-h|_{V}$.  As observed in \cite[Proof of Theorem 4.2]{Balwe-Sawant-R-eqivalence-IMRN}, $G_{L_i}$ is an isotropic group for each $i$.

For each $i$, let $K_i$ denote the field of fractions of the henselization $\mathcal O_{\A^1_k, p_i}^h$ of the local ring of $p_i$ at $\A^1_k$.  We set $W_i : = \Spec\left(\mathcal O_{\A^1_k, p_i}^h\right)$.  Note that $K_i$ is a separable extension of $k(t)$.  Since $G_{L_i}$ is isotropic, so is $G_{K_i}$ by \cite[Proposition A.2]{CTHHKPS}.  By \cite[Lemme 4.5]{Gille}, we have 
\[
G(K_i) = G(K_i)^+ \cdot G(\mathcal O_{\A^1_k, p_i}^h),
\]
so any morphism $\Spec(K_i) \to G$ is $\A^1$-chain homotopic to a morphism that factors as $\Spec(K_i) \xrightarrow{\eta_i} W_i \to G$, where $\eta_i:\Spec(K_i) \to W_i$ denotes the generic point of $W_i$.  Now, consider the pullback $h_V|_{W_i}: V \times_{\A^1_k} W_i \to G$ of $h_V$ and denote its restriction to $\Spec(K_i)$ by $h_i$.  By the above observation, $h_i: \Spec(K_i) \to G$ is $\A^1$-chain homotopic to a morphism that factors as 
\[
\Spec(K_i) \xrightarrow{\eta_i} W_i \xrightarrow{\rho_i} G,
\]
for some $\rho_i \in G(\mathcal O_{\A^1_k, p_i}^h)$.  Define $W : = \coprod_i W_i$ and $h_W:=\coprod \rho_i: W \to G$. 

We define $p_V: V \to \A^1_k$ to be the inclusion. For each $i$, let $p_i: W_i \to \A^1_k$ denote the natural map induced by $\Spec\left(\mathcal O_{\A^1_k, p_i}^h\right) \to \A^1_k$ and define $p_W:=\coprod_i p_i : W \to \A^1_k$.  By a standard limiting argument, we may replace $W_i$ with a scheme that is finite \'etale over $\A^1_k$.  We may further shrink $W_i$ suitably to ensure that 
\begin{itemize}
\item $W_i$ is an \'etale neighbourhood of $p_i$;
\item $\{p_V, p_W\}$ is an elementary Nisnevich cover of $\A^1_k$; and
\item the $\A^1$-chain homotopy between $h_i$ and $\rho_i \circ \eta_i$ spreads over $W_i$.
\end{itemize}
We then modify $W$ and $h_W$ accordingly.  We now apply \cite[Lemma 2.4]{Balwe-Sawant-R-eqivalence-IMRN} to conclude that $p$ and $q$ map to the same element in $\pi_0^{\A^1}(G)(k)$. 
\end{proof}

Combining Theorem \ref{theorem anisotropic R-equivalence} with the results of Asok-Hoyois-Wendt and Gille, we get the following generalizations of \cite[Theorem 3.6]{Balwe-Sawant-reductive} and \cite[Theorem 1]{Balwe-Sawant-reductive}, which we record below.

\begin{corollary}
\label{cor R-equivalence}
Let $k$ be an arbitrary field and let $G$ be a semisimple, simply connected algebraic group over $k$.  Then the natural map $\pi_0^{\A^1}(G)(k) \to G(k)/R$ is a bijection. 
\end{corollary}
\begin{proof}
As in \cite[Proof of Theorem 3.6]{Balwe-Sawant-reductive}, we are reduced to showing that $\pi_0^{\A^1}(G)(k) \to G(k)/R$ is a bijection, where $G$ is semisimple, simply connected and absolutely almost simple over $k$.  If $G$ is isotropic, then the desired result follows from \eqref{eq isotropic R-eqivalence}.  If $G$ is anisotropic, then we apply Theorem \ref{theorem anisotropic R-equivalence}.
\end{proof}

\begin{corollary}
Let $G$ be a reductive algebraic group over an arbitrary field $k$.  Then the following conditions are equivalent: 
\begin{enumerate}[label=$(\arabic*)$]
\item $\Sing G$ is $\A^1$-local;
\item $G$ is isotropic;
\item the presheaf $H^1_{\rm Nis}(-, G)$ is $\A^1$-invariant on smooth affine schemes over $k$.
\end{enumerate}
\end{corollary}
\begin{proof}
This follows exactly as in \cite[Proof of Theorem 1]{Balwe-Sawant-reductive}. 
\end{proof}

\begin{remark}
The referee asked us if the isomorphism given in Theorem \ref{theorem anisotropic R-equivalence} and Corollary \ref{cor R-equivalence} is induced by a morphism of sheaves with the following sheafified notion of $R$-equivalence. Define a sheaf $G/\@R$ of $R$-equivalence classes in an algebraic group $G$  to be the Nisnevich sheafification of the presheaf given by the colimit taken over the coequalizer of the diagram
\[
G^U \rightrightarrows G, 
\]
where $U$ varies over open subsets of $\A^1_k$ containing $0$ and $1$.  Let $G$ be a semisimple, simply connected algebraic group over an arbitrary field $k$.  The canonical epimorphism $G \to \pi_0^{A^1}(G)$ can be seen to factor through $G/\@R$ as follows.  For any open neighborhood $V$ of $U \times \{0,1\}$, if $\alpha \in G(V)$, then we claim that the restrictions $\alpha_0$ and $\alpha_1$ of $\alpha$ to $U \times \{0\}$ and $U \times \{1\}$ respectively map to the same element of $\pi_0^{\A^1}(G)(U)$.  Indeed, by \cite[Corollary 4.13]{Choudhury}, $\pi_0^{\A^1}(G)(U)$ injects into $\pi_0^{\A^1}(G)(k(U))$.  The images of $\alpha_0$ and $\alpha_1$ in $\pi_0^{\A^1}(G)(k(U))$ are $R$-equivalent and hence, map to the same element of $\pi_0^{\A^1}(G)(k(U))$, by Corollary \ref{cor R-equivalence}.  Thus, $\alpha_0$ and $\alpha_1$ map to the same element of $\pi_0^{\A^1}(G)(U)$, giving us a canonical morphism $G/\@R \to \pi_0^{\A^1}(G)$.  Corollary \ref{cor R-equivalence} says that this morphism is an isomorphism on sections over fields.  It is interesting to note that we need Corollary \ref{cor R-equivalence} to construct such a morphism of sheaves.
\end{remark}

\subsection{Contractions of \texorpdfstring{$\A^1$}{A1}-connected components of reductive groups} \hfill
\label{subsection contractions}

In this subsection, we record the generalizations of \cite[Th\'eor\`eme 5.8, Th\'eor\`eme 7.2]{Gille} and \cite[Corollary 4.3.6]{AHW2} to anisotropic and hence, to all semisimple, simply connected groups over any field.

\begin{proposition}
\label{prop k(t)}
Let $G$ be a semisimple, simply connected algebraic group over a field $k$.  Then the natural map 
\[
\pi_0^{\A^1}(G)(k) \to \pi_0^{\A^1}(G)(k(t))
\]
is an isomorphism.
\end{proposition}
\begin{proof}
Clearly, it suffices to prove the statement in the case where $G$ is absolutely almost simple.  By Theorem \ref{theorem anisotropic R-equivalence}, we see that the natural map
\[
\pi_0^{\A^1}(G)(F) \to G(F)/R
\]
is an isomorphism, for every field extension $F$ of $k$.  The proof in the case $G$ is isotropic follows from \cite[Th\'eor\`eme 5.8, Th\'eor\`eme 7.2]{Gille}.  Suppose that $G$ is anisotropic.  Note that the map $\pi_0^{\A^1}(G)(k) \to \pi_0^{\A^1}(G)(k(t))$ is injective.  For surjectivity, let $\alpha \in \pi_0^{\A^1}(G)(k(t)) = G(k(t))/R$.  Now, $\alpha$ corresponds to a rational map $f_\alpha: \A^1 \dashrightarrow G$.  Precomposing by the multiplication map $\A^1 \times \A^1 \to \A^1$, we get a rational map 
\[
h: \A^1 \times \A^1 \dashrightarrow G; \quad h(s,t) := f_\alpha(st). 
\]
Let us denote by $h_s$ the rational map $\A^1 \dashrightarrow G$ given by $t \mapsto h(s,t)$.  Then $h_0$ and $h_1$ are $R$-equivalent as elements of $G(k(t))$ and $h_0$ lies in the image of the map $\pi_0^{\A^1}(G)(k) \to \pi_0^{\A^1}(G)(k(t))$.  This completes the proof.
\end{proof}

\begin{remark}
It is worthwhile to remark that the results \cite[Th\'eor\`eme 5.8, Th\'eor\`eme 7.2]{Gille} that play a crucial role in the proof of isotropic case of Proposition \ref{prop k(t)} also play a crucial role in the proof of the anisotropic case of Theorem \ref{theorem anisotropic R-equivalence}.
\end{remark}

\begin{corollary}
\label{cor contraction}
If $G$ is a semisimple, simply connected algebraic group over a field $k$, then $\pi_0^{\A^1}(G)_{-1} = 0$.
\end{corollary}
\begin{proof}
We reduce to the case that $G$ is absolutely almost simple as above.  The result in the case where $G$ is isotropic can be found in \cite[Corollary 4.3.6]{AHW2}.  By \cite[Corollary 4.13]{Choudhury}, for any irreducible smooth scheme $X$ over $k$, the natural map $\pi_0^{\A^1}(G)(X) \to \pi_0^{\A^1}(G)(k(X))$ is injective.  The proof in the anisotropic case is immediate from Proposition \ref{prop k(t)}, since we have
\[
\pi_0^{\A^1}(G)_{-1}(U) = \ker \left(\pi_0^{\A^1}(G)(U \times \#G_m) \xrightarrow{(id,1)^*} \pi_0^{\A^1}(G)(U)\right),
\]
for any smooth $k$-scheme $U$ by definition.
\end{proof}

\subsection{Classifying spaces of reductive groups} \hfill
\label{subsection BG}

In this section, we briefly review the results of Elmanto, Kulkarni and Wendt \cite{EKW} that will be used in Section \ref{section strong}.  Let $Sch_k$ denote the site of schemes of finite type over $k$ with \'etale topology.  Let $\mathbf{R}_{\text{\'et}}$ denote the fibrant replacement functor for the \v{C}ech \'etale-local injective model structure.  The inclusion functor $i: Sm_k \to Sch_k$ induces a restriction functor $i^*$ from the category of simplicial \'etale sheaves on $Sch_k$ to the category of simplicial Nisnevich sheaves on $Sm_k$. 

Let $G$ be a reductive algebraic group over a field $k$.  We will denote by $BG$ the pointed simplicial sheaf whose $n$-simplices are $G^n$ with usual face and degeneracy maps.  Let $G$ also denote the underlying \'etale sheaf of groups on $Sch_k$. Then Morel and Voevodsky (see \cite[Page 130]{Morel-Voevodsky}) define $B_{\text{\'et}}G$ to be the simplicial Nisnevich sheaf defined by 
\[
B_{\text{\'et}}G := i^* \circ \mathbf{R}_{\text{\'et}}(BG). 
\]
The sheaf $\mathcal H^1_{\text{\'et}}(G)$ is defined to be the Nisnevich sheafification of the presheaf $U \mapsto H^1_{\text{\'et}}(U,G)$ on $Sm_k$.  We will freely use the identification $\mathcal H^1_{\text{\'et}}(G) = \pi_0^{\A^1}(B_{\text{\'et}} G)$ obtained in \cite[Theorem 1.5]{EKW}. 

We now discuss a variant of a result due to Elmanto, Kulkarni and Wendt \cite[Theorem 1.2, Proposition 3.7]{EKW}, which will play a key role in the proof of our main theorem.  We note below that under the assumption of simple connectedness, \cite[Proposition 3.7]{EKW} holds over an arbitrary field.  

\begin{theorem}
\label{theorem EKW}
Let $G$ be a semisimple, simply connected algebraic group over a field $k$.  Then the sheaf $\mathcal H^1_{\text{\'et}}(G)$ is $\A^1$-invariant.
\end{theorem}
\begin{proof}
The proof of \cite[Proposition 3.7]{EKW} works almost verbatim (with the use of \cite{Borel-Springer-II} in place of \cite{Steinberg}), so we only give an outline of the argument for the convenience of readers.  

By the strong Grothendieck-Serre property over a field (see \cite[Corollary 1]{Fedorov-Panin} and \cite[Corollary 1.2]{Panin}) along with \cite[Proposition 3.4]{EKW}, we are reduced to showing that the map $H^1_{\text{\'et}}(F, G) \to H^1_{\text{\'et}}(\A^1_F, G)$ induced by the projection map $\A^1_F \to \Spec(F)$ is bijective for any field extension $F$ of $k$.  Fix a separable closure $F_{\rm sep}$ of $F$.  By \cite[Theorem 1.1]{Raghunathan-Ramanathan} (restated as \cite[Theorem 3.6]{EKW}), there is a bijection
\[
H^1_{\text{\'et}}(F, G) \to \ker \left( H^1_{\text{\'et}}(\A^1_F, G) \to  H^1_{\text{\'et}}(\A^1_{F_{\rm sep}}, G)\right).
\]
It follows from \cite[Satz 3.3]{Harder-1967} that generically trivial $G$-torsors on smooth affine curves are trivial.  We now apply \cite[8.6]{Borel-Springer-II} to conclude that $H^1_{\text{\'et}}(\Spec(F_{\rm sep}(t)), G) = 0$, and consequently, every $G$-torsor on $\A^1_{F_{\rm sep}}$ is trivial. 
\end{proof}

\section{Semisimple, simply connected algebraic groups over a perfect field}
\label{section simply connected}

The aim of this section is to show that the sheaf of $\A^1$-connected components of a semisimple, simply connected algebraic group over a perfect field is birational.  The relationship between the sections of the sheaf of $\A^1$-connected components over a discrete valuation ring and its quotient field is tantamount to the understanding of the birationality property.  We treat the cases of isotropic and anisotropic groups separately below.

\begin{lemma}
\label{lemma dvr isotropic}
Let $G$ be an isotropic, semisimple, simply connected, absolutely almost simple algebraic group over a field $k$.  Let $R$ be a discrete valuation ring with quotient field $K$ and residue field $k$ such that $R$ is a $k$-algebra.  Then the restriction map $\pi_0^{\A^1}(G)(R) \to \pi_0^{\A^1}(G)(K)$ is an isomorphism.
\end{lemma}
\begin{proof}
Let $R^h$ denote the henselization of $R$.  Let $K^h$ denote the quotient field of $R^h$.  By \cite[Lemme 4.5]{Gille}, 
\[
G(K^h) = G(K^h)^+ \cdot G(R^h), 
\]
where $G(K^h)^+$ denotes the subgroup of $G(K^h)$ generated by unipotent elements.  This implies that the natural map $\pi_0^{\A^1}(G)(R^h) \to \pi_0^{\A^1}(G)(K^h)$ is an isomorphism.  We now have a diagram 
\[
\begin{xymatrix}{
\pi_0^{\A^1}(G)(R) \ar[r] \ar[d]& \pi_0^{\A^1}(G)(R^h) \ar@<-.5ex>[r] \ar@<.5ex>[r] \ar[d]^-{\simeq}& \pi_0^{\A^1}(G)(R^h \otimes_R R^h) \ar[d]\\ 
\pi_0^{\A^1}(G)(K) \ar[r] & \pi_0^{\A^1}(G)(K^h) \ar@<-.5ex>[r] \ar@<.5ex>[r] & \pi_0^{\A^1}(G)(K^h \otimes_K K^h) }
\end{xymatrix}
\]
in which the rows are equalizer diagrams, all the vertical maps are injective and the middle vertical map is an isomorphism.  Surjectivity of the leftmost vertical map now follows from an easy diagram chase.
\end{proof}

\begin{lemma}
\label{lemma anisotropic} 
Let $G$ be a semisimple algebraic group over a field $k$.  Let $R$ be an essentially smooth local $k$-algebra with quotient field $K$.  If $G_K$ is anisotropic, then the natural map $G(R) \to G(K)$ is an isomorphism.  Moreover, if $R$ is henselian, then the map $\pi_0^{\A^1}(G)(R) \to \pi_0^{\A^1}(G)(K)$ is an isomorphism.
\end{lemma}
\begin{proof}
First assume that $R$ is a discrete valuation ring.  Since $G$ is anisotropic, we have $G(R) = G(K)$ by \cite[Proposition A.3]{CTHHKPS}.  

In the general case, write $S = \Spec(R)$ and let $\alpha \in G(K)$.  For every $x \in S^{(1)}$, there exists a unique $\alpha_x \in G(\mathcal O_{S,x})$ such that $\alpha_x|_K = \alpha$.  Therefore, there exists an open subset $U$ of $S$ with $\codim(S\setminus U) \geq 2$ such that $\alpha$ lies in the image of the natural map $G(U) \to G(K)$.  This gives a rational map $\~\alpha: S \dashrightarrow G$, which is defined on all the codimension $1$ points of $S$.  But since $G$ is affine, the rational map $\~\alpha$ must be a morphism.  This can be seen by embedding $G$ into an affine space and reducing to the statement that any rational map from a normal variety to $\A^1$ defined at codimension $1$ points is a morphism.  The latter statement easily follows from \cite[Theorem 38]{Matsumura-CommAlg}, which says that any normal domain is the intersection of its localizations at height $1$ prime ideals.  Thus, we get an element $\~\alpha \in G(S)$ extending $\alpha \in G(K)$.

Assume now that $R$ is moreover henselian.  Injectivity of the map $\pi_0^{\A^1}(G)(R) \to \pi_0^{\A^1}(G)(K)$ follows from Lemma \ref{lemma pi0 injective}.  Surjectivity of the map $\pi_0^{\A^1}(G)(R) \to \pi_0^{\A^1}(G)(K)$ follows from the surjectivity of the maps $G(R) \to \pi_0^{\A^1}(G)(R)$ and $G(K) \to \pi_0^{\A^1}(G)(K)$.
\end{proof}

\begin{remark}
Statements related to those in Lemma \ref{lemma dvr isotropic} and Lemma \ref{lemma anisotropic} from the perspective of $R$-equivalence on group schemes have been studied in \cite[Section 8]{Gille-Stavrova}.  An independent proof of Lemma \ref{lemma dvr isotropic} can be given using \cite[Proposition 8.5]{Gille-Stavrova} in view of \eqref{eq isotropic R-eqivalence}.
\end{remark}

Another ingredient required for the proof of birationality of $\A^1$-connected components of a semisimple, simply connected group is the preservation of the birationality property under Weil restriction.

\begin{lemma}
\label{lemma weil restriction}
Let $L/k$ be a finite separable extension of fields and let $G$ be an algebraic group over $L$.  If $\pi_0^{\A^1}(G)$ is a birational sheaf, then so is $\pi_0^{\A^1}(R_{L/k}(G))$, where $R_{L/k}$ denotes the Weil restriction functor.
\end{lemma}
\begin{proof}
Let $X$ be a smooth scheme over $k$ and let $U$ be a dense open subscheme of $X$.  Without loss of generality we may assume that $X$ is henselian and $U = \Spec k(X)$.  We then have $U_L = \Spec K(X)\otimes_k L = \Spec L(X_L)$. We need to show that the restriction map $\pi_0^{\A^1}(R_{L/k}(G))(X) \to \pi_0^{\A^1}(R_{L/k}(G))(U)$ is a bijection.  It remains to show surjectivity of this map as this map is known to be injective (see \cite[Corollary 4.17]{Choudhury} or \cite[Lemma 3.45]{Morel-book}).  

Let $\alpha \in \pi_0^{\A^1}(R_{L/k}(G))(U)$.  Since $R_{L/k}(G) \to \pi_0^{\A^1}(R_{L/k}(G))$ is an epimorphism of Nisnevich sheaves, the element $\alpha$ is the image of an element $\~\alpha \in R_{L/k}(G)(U) = G(U_L)$.  Since $\pi_0^{\A^1}(G)$ is a birational sheaf, there exists $\beta \in G(X_L)$ such that $\beta|_{U_L}$ and $\~\alpha$ map to the same element in $\pi_0^{\A^1}(G)(U_L)$.  We will identify sections of $\pi_0^{\A^1}(G)$ over fields with $R$-equivalence classes, using \cite[Main Theorem]{Balwe-Sawant-R-eqivalence-IMRN} and \cite[Theorem 3.6]{Balwe-Sawant-reductive}.  Hence, $\beta|_{U_L}$ and $\~\alpha$ are $R$-equivalent.  So there exists a rational map $h: \P^1_L \dashrightarrow G$ defined at $0$ and $1$ such that $h(0) = \beta|_{U_L}$ and $h(1) = \~\alpha$.  This corresponds to a rational map $h': \P^1 \to R_{L/k}(G)$ connecting $\alpha$ to the restriction of the element in $R_{L/k}(G)(X)$ corresponding to $\beta \in G(X_L)$ by $R$-equivalence.  Consequently, the restriction map $\pi_0^{\A^1}(R_{L/k}(G))(X) \to \pi_0^{\A^1}(R_{L/k}(G))(U)$ is surjective.
\end{proof}

We are now set to prove the main theorem of this section.

\begin{theorem}
\label{theorem birational}
If $G$ is a semisimple, simply connected algebraic group over a perfect field $k$, then $\pi_0^{\A^1}(G)$ is a birational sheaf. 
\end{theorem} 
\begin{proof}
Since $G$ is the direct product of its almost simple factors, it suffices to prove the theorem in the case where $G$ is semisimple, simply connected and almost simple.  By \cite[Chapter 24.a]{Milne-algebraic groups} (see also \cite{Tits-classification}), there exists a finite separable extension $L/k$ and an absolutely almost simple group $G'$ such that $G$ is isomorphic to $R_{L/k}(G')$ as an algebraic group over $k$.  In view of Lemma \ref{lemma weil restriction}, it suffices to prove that $\pi_0^{\A^1}(G')$ is a birational sheaf.

In view of Lemma \ref{lemma henselian reduction}, it suffices to show that for every henselian local $k$-scheme $S = \Spec(R)$, the restriction map $\pi_0^{\A^1}(G')(S) \to \pi_0^{\A^1}(G')(k(S))$ is a surjection (injectivity of this map follows from Lemma \ref{lemma pi0 injective}).  Let $\ell$ denote the residue field of $S$.  Since $k$ is perfect, there exists an inclusion $\ell \hookrightarrow R$ as $k$-algebras.  Hence, $\pi_0^{\A^1}(G')(S) = \pi_0^{\A^1}(G'_\ell)(S)$.  Thus, without loss of generality, we may assume that $\ell = k$.  If $G'$ is anisotropic, then the desired statement is given by Lemma \ref{lemma anisotropic}.

If $G'$ is isotropic, then $\pi_0^{\A^1}(G')$ is strongly $\A^1$-invariant by \cite[Theorem 4.3.3]{AHW2}.  In particular, $\pi_0^{\A^1}(G')$ is an unramified sheaf.  Hence, 
\begin{equation}
\label{eqn unramified}
\pi_0^{\A^1}(G')(S) = \bigcap_{x \in S^{(1)}} \pi_0^{\A^1}(G')(\mathcal O_{S,x}) \subseteq \pi_0^{\A^1}(G')(k(S)).  
\end{equation}
Let $\alpha \in \pi_0^{\A^1}(G')(k(S))$.  We will show that $\alpha$ ``extends to all the codimension $1$ points of $S$''.  Let $x \in S^{(1)}$.  By Lemma \ref{lemma dvr isotropic} applied to $\Spec (\mathcal O_{S,x})$,  there exists $\alpha_x \in \pi_0^{\A^1}(G')(\mathcal O_{S,x})$ such that $\alpha_x|_{k(S)} = \alpha$.  Since the restriction map $\pi_0^{\A^1}(G')(\mathcal O_{S,x}) \to \pi_0^{\A^1}(G')(k(S))$ is injective for every $x \in S^{(1)}$, it follows that $\alpha_x|_{k(S)} = \alpha_y|_{k(S)}$ for all $x, y \in S^{(1)}$.  The surjectivity of $\pi_0^{\A^1}(G')(S) \to \pi_0^{\A^1}(G')(k(S))$ now follows from \eqref{eqn unramified}.
\end{proof}

\begin{corollary}
\label{cor ss sc}
If $G$ is a semisimple, simply connected algebraic group over a perfect field $k$, then $\pi_0^{\A^1}(G)$ is strongly $\A^1$-invariant. 
\end{corollary}
\begin{proof}
This is a straightforward consequence of Theorem \ref{theorem birational} and Remark \ref{rmk birationality and strong a1-invariance}, since $\pi_0^{\A^1}(G)$ is an $\A^1$-invariant sheaf by \cite[Theorem 4.18]{Choudhury}.
\end{proof}

\section{Strong \texorpdfstring{$\A^1$}{A1}-invariance of the sheaf of \texorpdfstring{$\A^1$}{A1}-connected components of a reductive group}
\label{section strong}

In this section, we build on the results of the previous section to show that the sheaf of $\A^1$-connected components of any reductive algebraic group over a perfect field is strongly $\A^1$-invariant.  We begin by observing that strong $\A^1$-invariance of $\A^1$-connected components is preserved under isogenies. 

\begin{lemma}
\label{lemma isogeny}
Let $k$ be a field.  Let $\~G \to G$ be a central isogeny of semisimple algebraic groups, the kernel of which is a group $\mu$ of multiplicative type.  If $\pi_0^{\A^1}(\~G)$ is strongly $\A^1$-invariant, then so is $\pi_0^{\A^1}(G)$.
\end{lemma}
\begin{proof}
Note that $\~G \to G$ is an \'etale locally trivial $\mu$-torsor, classified by a map $G \to B\mu$ in the simplicial homotopy category by \cite[Page 130, \textsection 4, Proposition 1.15]{Morel-Voevodsky}.  Thus, there is a simplicial fiber sequence of simplicial sheaves of sets on $(Sm_k)_{\text{\'et}}$ of the form
\[
\~G \to G \to B \mu \to B\~G \to BG. 
\]
More precisely, the map $G \to B \mu$ is simplicially a homotopy principal $\~G$-fibration.  Let $\pi: (Sm_k)_{\text{\'et}} \to (Sm_k)_{\text{Nis}}$ denote the canonical morphism of sites.  The space $B_{\text{\'et}} \mu$ is obtained by applying the right Quillen functor $\pi_*$ to a fibrant replacement $\~{B\mu}$ of $B\mu$ for the \v{C}ech \'etale-local injective model structure.  As $\pi_*$ is a right Quillen functor, it preserves homotopy principal fibrations.  It can then be verified that $\~G$ is weakly equivalent to the fiber of $G \to B_{\text{\'et}} \mu$.  We thus have a simplicial fiber sequence
\[
\~G \to G \to B_{\text{\'et}} \mu.
\]
By \cite[\textsection 4.3, Proposition 3.1]{Morel-Voevodsky}, the space $B_{\text{\'et}} \mu$ is $\A^1$-local.  Since $\pi_0^{\A^1}(\~G)$ is strongly $\A^1$-invariant by hypothesis, by \cite[Theorem 6.50]{Morel-book}, the sequence is an $\A^1$-fiber sequence.  This gives rise to an exact sequence of sheaves of sets
\[
\cdots \to \pi_0^{\A^1}(\~G) \to \pi_0^{\A^1}(G) \to \pi_0^{\A^1}(B_{\text{\'et}} \mu).
\]
For every smooth henselian local $k$-scheme $U$, we have a commutative diagram
\[
\begin{xymatrix}{
G(U) \ar[r] \ar@{->>}[d] & H^1_{\text{\'et}}(U, \mu) \ar[r] \ar[d]^-{\simeq} & H^1_{\text{\'et}}(U, \~G) \ar[r] \ar[d]^-{\simeq} & H^1_{\text{\'et}}(U, G) \\
\pi_0^{\A^1}(G)(U) \ar[r]  & \pi_0^{\A^1}(B_{\text{\'et}}\mu)(U) \ar[r]  & \pi_0^{\A^1}(B_{\text{\'et}}\~G)(U) & } 
\end{xymatrix}
\]
in which the top row is exact.  Note that the cokernel of the map $\pi_0^{\A^1}(G)(U) \to \pi_0^{\A^1}(B_{\text{\'et}}\mu)(U)$ is contained in $\pi_0^{\A^1}(B_{\text{\'et}}\mu)(U) \simeq H^1_{\text{\'et}}(U, \~G)$.  By Theorem \ref{theorem EKW}, the sheaf $\pi_0^{\A^1}(B_{\text{\'et}}\~G) = \mathcal H^1_{\text{\'et}}(\~G)$ is $\A^1$-invariant.  Thus, the cokernel of $\pi_0^{\A^1}(G) \to \pi_0^{\A^1}(B_{\text{\'et}}\mu)$ is an $\A^1$-invariant sheaf.  Note that the sheaf $\pi_0^{\A^1}(B_{\text{\'et}}\mu) = \mathcal H^1_{\text{\'et}}(\mu)$ is strongly $\A^1$-invariant by \cite[Example 4.6]{Asok-crelle}.  Applying \cite[Theorem 1.3]{Choudhury-Hogadi}, we see that $\coker\left(\pi_0^{\A^1}(G) \to \pi_0^{\A^1}(B_{\text{\'et}}\mu)\right)$ is a strongly $\A^1$-invariant sheaf.  Now \cite[Theorem 1.5]{Choudhury-Hogadi} implies that $\im\left(\pi_0^{\A^1}(G) \to \pi_0^{\A^1}(B_{\text{\'et}}\mu)\right)$ is a strongly $\A^1$-invariant sheaf.  Since $\pi_0^{\A^1}(\~G)$ is strongly $\A^1$-invariant, another application of \cite[Theorem 1.3, Theorem 1.5]{Choudhury-Hogadi} yields that $\im\left(\pi_0^{\A^1}(\~G) \to \pi_0^{\A^1}(G)\right)$ is strongly $\A^1$-invariant.  We now conclude from the short exact sequence
\[
1 \to \im\left(\pi_0^{\A^1}(\~G) \to \pi_0^{\A^1}(G)\right) \to \pi_0^{\A^1}(G) \to \im\left(\pi_0^{\A^1}(G) \to \pi_0^{\A^1}(B_{\text{\'et}}\mu)\right) \to 1
\]
that $\pi_0^{\A^1}(G)$ is strongly $\A^1$-invariant.
\end{proof}

We now use the standard structure theory of reductive groups along with Lemma \ref{lemma isogeny} and the results of Section \ref{section simply connected} to obtain our main theorem.

\begin{theorem}
\label{theorem reductive}
If $G$ is a reductive algebraic group over a perfect field $k$, then $\pi_0^{\A^1}(G)$ is strongly $\A^1$-invariant.  
%If $G$ is a reductive algebraic group over field $k$ of characteristic $0$, then $\pi_0^{\A^1}(G)$ is strongly $\A^1$-invariant.  
\end{theorem}
\begin{proof}
Let $G_{\rm der}$ denote the derived group of $G$, which is a normal, semisimple subgroup scheme of $G$.  Let $G_{\rm sc}$ denote the universal simply connected cover of $G_{\rm der}$.  By Corollary \ref{cor ss sc}, the sheaf $\pi_0^{\A^1}(G_{\rm sc})$ is strongly $\A^1$-invariant.  Lemma \ref{lemma isogeny} then implies that $\pi_0^{\A^1}(G_{\rm der})$ is strongly $\A^1$-invariant.

There exists a central isogeny (see \cite[Expos\'e XXII, Proposition 6.2.4]{SGA3.3}) $G_{\rm der} \times T \to G$, where $T$ is a torus, the radical of $G$.  This is a faithfully flat, finitely presented morphism, whose kernel is a finite group of multiplicative type contained in the center of $G_{\rm der} \times T$.  It follows by another application of Lemma \ref{lemma isogeny} that $\pi_0^{\A^1}(G)$ is strongly $\A^1$-invariant, since $\pi_0^{\A^1}(G_{\rm der} \times T) \simeq \pi_0^{\A^1}(G_{\rm der}) \times T$ is strongly $\A^1$-invariant.
\end{proof}

A straightforward consequence of Theorem \ref{theorem reductive} and \cite[Theorem 6.50]{Morel-book} is the following application to $\A^1$-fiber sequences.

\begin{corollary}
\label{cor a1 fiber sequence}
Let $G$ be a reductive algebraic group over a perfect field $k$.  For any Nisnevich locally trivial $G$-torsor $\mathcal P \to \mathcal X$ in $\mathcal H(k)$, 
\[
\mathcal P \to \mathcal X \to BG 
\]
is an $\A^1$-fiber sequence.
\end{corollary}

\end{document}